\documentclass[12pt]{amsart} 
\usepackage{geometry}             

\usepackage[mathscr]{euscript}
\usepackage[T1]{fontenc}
\usepackage{amsfonts}

\usepackage{amssymb}
\usepackage{amsmath}
\usepackage{amsthm}
\usepackage{mathrsfs}
\usepackage{upgreek}

\usepackage[dvipsnames,usenames]{xcolor}
\usepackage[colorlinks=true,urlcolor=blue,linkcolor=blue,citecolor=blue]{hyperref}
\hypersetup{
	linkbordercolor={1 0 0}, 
	citebordercolor={0 1 0} 
}
\usepackage{cite}
\usepackage[noabbrev]{cleveref}

\usepackage{color}
\usepackage{subfig}
\usepackage{tikz} 
 
\newcommand{\al}{\alpha}
\newcommand{\be}{\beta}
\newcommand{\ga}{\gamma}

\newcommand{\ep}{\varepsilon}

\newcommand{\la}{\lambda}
\renewcommand{\phi}{\varphi}
\newcommand{\si}{\sigma}

\newcommand{\Ga}{\Gamma}

\newcommand{\Si}{\Sigma}
\newcommand{\ZZ}{{\mathbb Z}}

\newcommand{\RR}{{\mathbb R}}

\newcommand{\cG}{\mathcal G}

\newcommand{\lto}{\longrightarrow}
\newcommand{\lk}{\operatorname{\ell{\mathit k}}}

\newcommand{\sig}{\operatorname{sig}}

\newcommand{\Int}{\operatorname{int}}
\newcommand{\lb}{\!\! \left\bracevert \!}
\newcommand{\rb}{\! \right\bracevert \!\!\!}

\newcommand{\sm}{\smallsetminus}
\newcommand{\co}{\colon}

\newcommand*\wbar[1]{
  \hbox{ \kern-0.2em%
    \vbox{%
      \hrule height 0.5pt  
      \kern0.25ex
      \hbox{%
        \kern-0.15em
        \ensuremath{#1}%
        \kern-0.05em
      }%
    }%
  \kern0.05em}%
}

\newtheorem{theorem}{Theorem}
\newtheorem{lemma}[theorem]{Lemma}
\newtheorem{proposition}[theorem]{Proposition}

\newtheorem{corollary}[theorem]{Corollary}

\theoremstyle{definition}     
\newtheorem{definition}[theorem]{Definition}

\theoremstyle{remark}
\newtheorem{remark}[theorem]{Remark}

\title[A characterization of alternating links in thickened surfaces]{A characterization of alternating links in thickened surfaces}
\author[H. U. Boden]{Hans U. Boden}
\address{Mathematics \& Statistics, McMaster University, Hamilton, Ontario}
\email{boden@mcmaster.ca}
\thanks{The first author was partially funded by the Natural Sciences and Engineering Research Council of Canada.}

\author[H. Karimi]{Homayun Karimi}
\address{Mathematics \& Statistics, McMaster University, Hamilton, Ontario}
\email{karimih@math.mcmaster.ca}

\subjclass[2020]{Primary: 57K10, Secondary: 57K12}
\keywords{Links in thickened surfaces, alternating link, virtual link, checkerboard coloring, spanning surface, Gordon-Litherland pairing, definite surface}

\begin{document}

\begin{abstract}
We use an extension of Gordon-Litherland pairing to thickened surfaces to give a topological characterization of alternating links in thickened surfaces. If $\Si$ is a closed oriented surface and 
$F$ is a compact unoriented surface in $\Si \times I$, then the Gordon-Litherland pairing defines a symmetric bilinear pairing on the first homology of $F$. 
A compact surface in $\Si \times I$ is called \textit{definite} if its Gordon-Litherland pairing is a definite form. We prove that a link $L$ in a thickened surface is non-split, alternating, and of minimal genus if and only if it bounds two definite surfaces of opposite sign. 
\end{abstract}

\maketitle
\setcounter{section}{1} \noindent
\subsection{Introduction} \label{S1}  
Alternating links in $S^3$ can be characterized as precisely those links which simultaneously bound both positive and negative definite spanning surfaces. This beautiful result was established recently  
by Greene in \cite{Greene}, and Howie obtained a similar characterization of alternating links in $S^3$ in terms of spanning surfaces in \cite{Howie}. 

These results have been extended to almost alternating knots by Ito in \cite{Ito-2018}, and to toroidally alternating knots by Kim in \cite{Kim-2019}.

In \cite{Greene}, Greene showed that, in general, $Y=S^3$ is the only $\ZZ/2$ homology 3-sphere containing a link that bounds both positive and negative definite surfaces (and that any such link in $S^3$ is alternating). 

In this paper, we study links in thickened surfaces, and we present a generalization of the results of Greene and Howie, giving a topological characterization of alternating links in thickened surfaces.

Let $\Si$ be a compact, connected, oriented surface and $I = [0,1]$, the unit interval.  
A link in $\Si \times I$ is an embedding $L \colon \bigsqcup_{i=1}^m S^1 \hookrightarrow \Si \times I$. In the following, we identify this embedding with its image and consider links $L \subset \Si \times I$ up to orientation-preserving homeomorphisms of the pair $(\Si \times I, \Si \times \{0\})$. 

A link $L$ in a thickened surface $\Si \times I$ can be represented by its link diagram, which is the tetravalent graph $D$ on $\Si$ obtained under projection $p\co \Si \times I \to \Si$. The arcs of $D$ are drawn to indicate over and under crossings near vertices in the usual way. The link diagram $D$ on $\Si$ is said to be \textit{alternating} if its crossings alternate from over to under around each component of the link. A link $L$ in  $\Si \times I$ is  alternating if it admits an alternating link diagram on $\Si$ (see \Cref{figure-1}).

\begin{figure}[ht]
\centering
\includegraphics[scale=1.00]{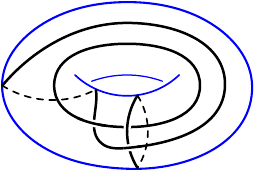}\hspace{.8cm} 
\includegraphics[scale=1.00]{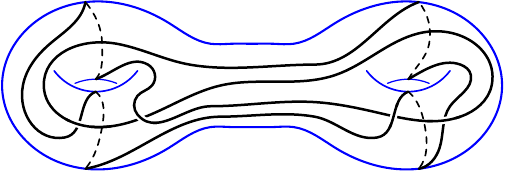} 
\vspace{-1mm}
\caption{\small An alternating link in the thickened torus and an alternating knot in a thickened surface of genus two.}\label{figure-1}
\end{figure}

Given a compact unoriented surface $F$ in $S^3$, Gordon and Litherland defined a symmetric bilinear pairing on $H_1(F)$. In the case that $F$ is a spanning surface for a link $L \subset S^3$, the signature of $L$ can be computed in terms of the signature of the pairing together with a correction term. 
The Gordon-Litherland pairing is extended to $\ZZ/2$ homology 3-spheres in \cite{Greene} and to thickened surfaces $\Si \times I$ in \cite{Boden-Chrisman-Karimi-2021}. Our characterization of alternating links in $\Si \times I$ is phrased in terms of the Gordon-Litherland pairing on spanning surfaces for the link.

In this paper, we are mainly interested in links in $\Si \times I$ which are $\ZZ/2$-null-homologous, or equivalently links in $\Si \times I$ that admit spanning surfaces. A  link $L$ in $\Si \times I$ is said to be \textit{split} if it can be represented by a disconnected diagram $D$ on $\Si$.
A link diagram $D$ on $\Si$ is said to be \textit{cellularly embedded} if $\Si \sm D$ is a union of disks.

Let $L$ be a link in $\Si \times I$ with alternating diagram $D$. 
We further assume that $D$ is cellularly embedded.
This implies that $D$ is connected, and therefore that $L$ is non-split (see \Cref{cor-min-g}).
It follows that the complementary regions of $\Si \sm D$ admit a checkerboard coloring, and that the black and white regions form spanning surfaces for $L$ which we denote $B$ and $W$. A straightforward argument will show that $B$ is negative definite and $W$ is positive definite with respect to their associated Gordon-Litherland pairings (see \Cref{minimal-genus}).  

Our main result is a converse to this statement given by the following theorem. The definition of minimal genus is given below.

\begin{theorem} \label{thm:main}
Let $L$ be a  link in $\Si\times I$, and assume that $L$ bounds a positive definite spanning surface and a negative definite spanning surface. Then $L \subset \Si \times I$  is  a non-split, alternating link of minimal genus.
\end{theorem}

\Cref{thm:main} applies to give a characterization of alternating virtual links. 

Virtual links can be defined as virtual link diagrams up to the generalized Reidemeister moves \cite{KVKT}. 
One can also define them as stable equivalence classes of links in thickened surfaces.
Here, two links $L_0 \subset \Si_0\times I$ and $L_1 \subset \Si_1 \times I$ are said to be \textit{stably equivalent} if one is obtained from the other by a finite sequence of isotopies, homeomorphisms,\footnote{Here homeomorphism means orientation-preserving homeomorphisms of the pair $(\Si \times I, \Si \times \{0\}).$} 
 stabilizations, and destabilizations. We take a moment to explain stabilization.

Given a link $L \subset \Si \times I$, let $h \co S^0 \times D^2 \to \Si$ be the attaching region for a 1-handle which is disjoint from the image of $L$ under projection $p\co \Si \times I \to \Si$. Let 
$$\Si' = \Si \sm h(S^0 \times D^2) \cup_{S^0 \times S^1} D^1 \times S^1,$$ and let $L'$ be the image of $L$ in $\Si'\times I$ under the inclusions
$$L\hookrightarrow(\Si\sm h(S^0\times D^2))\times I\hookrightarrow\Si'\times I.$$
Then we say that $(\Si' \times I,L')$ is obtained from $(\Si \times I, L)$ by stabilization, and 
destabilization is the opposite procedure.

In \cite{Carter-Kamada-Saito}, Carter, Kamada, and Saito give a one-to-one correspondence between virtual links and stable equivalence classes of links in thickened surfaces.

The virtual genus of a virtual link is the minimum genus over all surfaces $\Si$ such that
$\Si \times I$ contains a representative for the virtual link. In that case, we say that the representative $L \subset \Si \times I$ has \textit{minimal genus}. If $L$ is non-split, and it has minimal genus, then any diagram $D \subset \Si$ for $L$ is necessarily cellularly embedded, since otherwise destabilization would produce a representative of the same virtual link on a lower genus surface. Kuperberg's theorem shows that every non-split virtual link has an irreducible representative which is unique up to orientation-preserving homeomorphism of the pair $(\Si \times I, \Si \times \{0\})$ \cite{Kuperberg}. In particular, it implies that two minimal genus representatives of the same non-split virtual link are equivalent under orientation-preserving homeomorphism of the pair $(\Si \times I, \Si \times \{0\}).$

We combine the results to give the following useful characterization of alternating virtual links.

\begin{corollary} \label{cor:mainv}
A virtual link is non-split  and alternating if and only if it admits a representative $L$ in $\Si \times I$ which bounds a positive definite spanning surface and a negative definite spanning surface. 
\end{corollary}

\medskip\noindent
{\textbf{Notations and Conventions.}} Homology groups are taken with $\ZZ$ coefficients unless otherwise indicated.
Links in thickened surfaces are assumed to lie in the interior, and spanning surfaces are assumed to be connected but not necessarily oriented. 

\setcounter{section}{2} \noindent
\subsection{Gordon-Litherland pairing} \label{S2} 
In this section, we review the Gordon-Litherland pairing \cite{GL-1978} and its extension to links in thickened surfaces \cite{Boden-Chrisman-Karimi-2021}.
The pairing is defined for any link $L \subset \Si \times I$ that admits a \textit{spanning surface}, which is a compact, connected surface $F$ embedded in $\Si \times I$ with boundary $\partial F =L$. The surface $F$ may or may not be orientable, and here we consider it as an \textit{unoriented} surface. Not all links $L \subset \Si \times I$ admit spanning surfaces, and in fact Proposition 1.7 of \cite{Boden-Karimi} implies that $L$ admits a spanning surface if and only if $[L]$ is trivial in $H_1(\Si\times I;\ZZ/2).$

The link diagram $D$ is the decorated graph on $\Si$ obtained as the image of $L$ under the projection
$p\co \Si \times I \to \Si$. Then $D$ is called \textit{checkerboard colorable} if the complementary regions of $\Si \sm D$ can be colored black and white so that, whenever two regions share an edge, one is white and the other is black. A link $L$ in $\Si \times I$ is said to be checkerboard colorable if it admits a diagram which is checkerboard colorable. If $L$ is non-split and admits a checkerboard colored diagram, then the black and white regions determine unoriented spanning surfaces which we call \textit{checkerboard surfaces}. A straightforward argument shows that a link in $\Si \times I$ is checkerboard colorable if and only if it bounds an unoriented spanning surface.

Next, we recall the definition of the asymmetric linking for simple closed curves in a thickened surface.
Let $\Si$ be a compact, connected, oriented surface. The asymmetric linking pairing in $\Si \times I$ is taken relative to $\Si \times \{1\}$ and defined as follows. 
Given an oriented simple closed curve $J$ in the interior of $\Si \times I,$ then by Proposition 7.1 of \cite{Boden-Gaudreau-Harper-2016}, $H_1(\Si \times I \sm J, \Si \times \{1\})$ is infinite cyclic and
generated by a meridian $\mu$ of $J$. The meridian $\mu$ here is oriented as the boundary of a small oriented 2-disk which intersects $J$ transversely at one point with oriented intersection number equal to one.

If $K$ is an oriented  simple closed curve in the interior of $\Si\times I$ and disjoint from $J$, then define $\lk(J,K)$ to be the unique integer $m$ such that $[K] = m [\mu]$ in  $H_1(\Si \times I \sm J, \Si \times \{1\})$.
Alternatively, if $B$ is a 2-chain in $\Si \times I$ such that $\partial B = K -v$, where $v$ is a 1-cycle in $\Si \times \{1\}$,
then $\lk(J,K) = J \cdot B$, where $\cdot$ denotes the intersection number.

One can determine the asymmetric linking numbers easily using the following simple diagrammatic description.
If $J,K$ are two oriented disjoint simple closed curves
in $\Si\times I$, and $J\cup K$ is represented as a diagram on $\Si$, then $\lk(J,K)$ is equal to the number of times $J$ goes above $K$ with sign given by comparing with orientation of $\Si$, where ``above'' refers to the positive $I$ direction in $\Si\times I$. 

For example, the linking numbers for the link in \Cref{figure-2} are given by $\lk(J,K)=0$, $\lk(K,J)=-1$, $\lk(J,L)=-1$, $\lk(L,J)=0$, $\lk(K,L)=0$, and $\lk(L,K)=1.$

\begin{figure}[ht]
\centering
\includegraphics[scale=0.80]{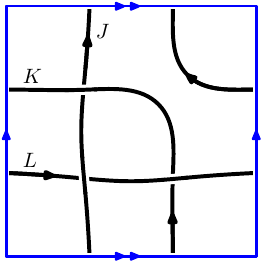} 
\vspace{-2mm}
\caption{\small  An alternating link in the thickened torus.}
\label{figure-2}
\end{figure}

Now suppose that $F$ is a compact, connected, unoriented surface embedded in $\Si\times I$. Its normal bundle $N(F)$ has boundary a $\{\pm 1\}$-bundle $\widetilde{F}\stackrel{\pi}{\lto}F$, a double cover with $\widetilde{F}$ oriented. (If $F$ is oriented, then $\widetilde{F}$ is the trivial double cover.) 
Define the transfer map $\tau \co H_1(F) \to H_1(\widetilde{F})$ by setting $\tau([a]) = [\pi^{-1}(a)].$

The Gordon-Litherland pairing is extended to thickened surfaces $\Si \times I$ in \cite{Boden-Chrisman-Karimi-2021},  and we review its  definition. 
Let $F \subset \Si \times I$ be a compact, unoriented surface without closed components. For $a,b\in H_{1}(F)$, let  $\cG_F(a,b)=\tfrac{1}{2}\big(\lk(\tau a,b)+\lk(\tau b,a)\big).$ 
(This is a slightly different formulation from that in \cite{Boden-Chrisman-Karimi-2021}, but the two are equivalent.)

To see the pairing is well-defined, we make two elementary observations. First, if two oriented curves $K,K'$ are homologous in $\Si \times I \sm J$, then $\lk(J,K) = \lk(J,K')$. Second, if $a,a',b \in H_1(F)$ and if $a$ and $a'$ are homologous on $F$, then $\tau a$ and $\tau a'$ are homologous in $\Si \times I \sm b$. 

Thus, the pairing is well-defined. It is clearly symmetric by definition, but it is not entirely clear that it is integral. To see that, consider the local contributions to $\cG_F(a,b)$ anywhere $a$ and $b$ cross or intersect. If $a$ crosses over $b$, then the local contribution to $\lk(\tau b,a)$ is $0$, and the local contribution to $\lk(\tau a,b)$, is $2\ep$, where $\ep=\pm 1$ is determined by the orientations of $a$ and $b$. Thus the contribution of the crossing to $\cG_F(a,b)$ is  $\ep = \pm 1$. (If $F$ is a checkerboard surface and $a,b$ are simple closed curves on $F$, then they can cross only along a twisted band as they pass through a crossing of the diagram (see \Cref{figure-local} (left)).) By symmetry, the same is true if $b$ crosses over $a$. Now suppose that $a$ and $b$ intersect at a point (see \Cref{figure-local} (right)). One can check that the local contributions of that point to $\lk(\tau a,b)$ and $\lk(\tau b,a)$ cancel, so that the points where $a$ and $b$ intersect do not contribute to $\cG_F(a,b)\in \ZZ$. Therefore,
$$\cG_{F}\colon H_{1}(F)\times H_{1}(F)\to\ZZ,$$
gives a well-defined symmetric bilinear pairing.

\begin{figure}[ht]
\centering
\includegraphics[scale=0.90]{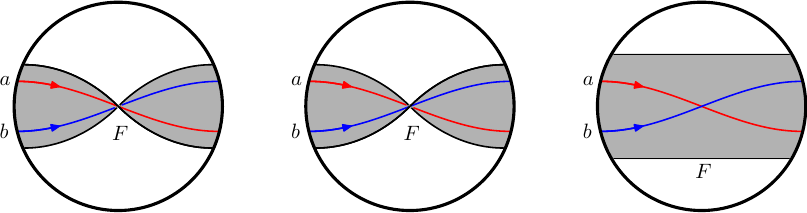} 
\vspace{-2mm}
\caption{\small   
On the left, two curves $a,b$ cross a right-handed half-twist in $F$ with $a$ crossing over $b$, and the local contribution to $\cG_F(a,b)$ is 1. 
In the middle,  $a,b$ cross a left-handed half-twist in $F$ with $b$ crossing over $a$, and the local contribution to $\cG_F(a,b)$ is $-1.$
On the right, $a,b$ intersect on $F$, and the local contribution to $\cG_F(a,b)$ is 0.}
\label{figure-local}
\end{figure}

For $x \in H_1(F),$ let $\lb x^{}_{} \rb_F = \cG_{F}(x,x)$. Clearly $\lb -x \rb_F = \lb x^{}_{} \rb_F$. The number $\frac{1}{2}\lb x^{}_{} \rb_F$ is called the \textit{framing} of $x$ in $F$.

\begin{figure}[ht]
\centering
\includegraphics[scale=0.90]{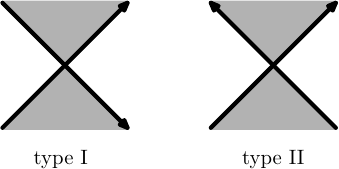} \qquad \qquad
\includegraphics[scale=0.90]{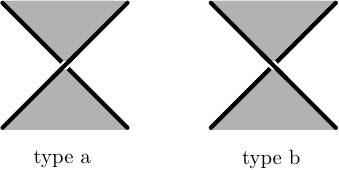} 
\vspace{-2mm}
\caption{\small  Type I and II crossings (left), and incidence numbers (right).}\label{figure-3}
\end{figure}

Assume now that $L \subset \Si \times I$ is a link with $m$ components and write $L = K_1 \cup \cdots \cup K_m$.
Suppose further that $F \subset \Si \times I$ is a spanning surface for $L$.
Each component represents an element $[K_i] \in H_1(F)$, well-defined up to sign, and $\frac{1}{2}\lb [K_i]\rb_F$ is equal to the framing that the surface $F$ induces on $K_i$. Set $e(F) = -\frac{1}{2} \sum_{i=1}^m \lb [K_i] \rb_F$, the Euler number of $F$.
It is equal to the self-intersection number of $F$, which is defined as a sum of signs $\ep_p$ over points $p \in F \cap F'$, where $F'$ is a transverse pushoff $F'$ of $F$ in $\Si \times I \times I$. The sign $\ep_p$ is computed by choosing a local orientation of $F$ at $p$ and using it to induce a local orientation on $F'$ at $p$. Then $\ep_p=\pm 1$, and it is determined by comparing the orientations of $F$ and $F'$ at $p$ with the given orientation on $\Si \times I \times I$. The sign $\ep_p$ is independent of the choice of local orientation of $F$.

Further, set   $e(F, L) = -\frac{1}{2} \lb [L] \rb_F.$ 
If $L'= K_1' \cup \cdots \cup K_m'$ denotes the push-off of $L$ that misses $F$, then it follows that 
\begin{eqnarray*}
e(F) &=& -\sum_{i=1}^m \lk(K_i, K_i'),\\
e(F,L) &=&  -\lk(L,L')= -\sum_{i,j=1}^m \lk(K_i, K_j').
\end{eqnarray*}

Note that $e(F)$ is independent of the choice of orientation on $L$, whereas $e(F,L)$ is not.  
To see this, notice that in the formula for $e(F)$, each component $K_i$ can be oriented arbitrarily provided its push-off $K_i'$ is oriented in a compatible manner.  
The two quantities are related by the formula
\begin{equation} \label{eqn:Euler}
e(F,L) = e(F) - \la(L),
\end{equation}
where $\la(L) = \sum_{i \neq j} \lk(K_i, K_j)$ denotes the total linking number of $L$.\footnote{In \cite{Greene}, the total linking is defined as $\lk(L) = \sum_{i < j} \lk(K_i, K_j)$.
Our formula is different since $\lk(K_i,K_j) \neq \lk(K_j, K_i)$ in general for links in $\Si \times I$. If $L$ is classical, then $\la(L) = 2\lk(L)$. (Note that there is a missing factor of 2 in the formula for the Euler numbers in \cite[p.2137]{Greene}.)}

Two spanning surfaces for a given link are said to be \textit{$S^*$-equivalent} if one can be obtained from the other by ambient isotopy, attachment or removal of tubes, and attachment or removal of a small half-twisted band, as depicted in \Cref{figure-4}.

\begin{figure}[ht]
\centering
\includegraphics[scale=1.40]{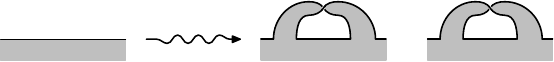} 
\vspace{-1mm}
\caption{\small  Attaching a small half-twisted band.}\label{figure-4}
\end{figure}

The signature of the pair $(L,F)$ is defined by setting $\si_{F}(L)=\sig(\cG_{F})+\frac{1}{2}e(F,L)$. 
The following lemma shows that  $\si_{F}(L)$ gives a well-defined invariant of the pair $(F,L)$ depending only on the $S^*$-equivalence class of $F$. For a proof, see \cite[Lemma 2.3]{Boden-Chrisman-Karimi-2021}.

\begin{lemma} \label{GL-formula}
If $F$ and $F'$ are $S^{*}$-equivalent spanning surfaces for a link $L$, then $\si_{F}(L)=\si_{F'}(L)$.
\end{lemma} 

For classical links, every link admits a spanning surface, and any two are $S^*$-equivalent (for a proof, see \cite[Theorem 11]{GL-1978} and \cite{Yasuhara-2014}). 
For links in thickened surfaces, the situation is more complicated.

In general, a link $L \subset \Si \times I$ admits  a spanning surface if and only if it is checkerboard colorable. If $g(\Si)>0$, not all links in $L \subset \Si \times I$ are checkerboard colorable (see \Cref{figure-5}).
Furthermore, not all spanning surfaces for a link $L \subset \Si \times I$ will be $S^*$-equivalent. In fact,
if $L$ is non-split, then two spanning surfaces $F$ and $F'$ for $L$ are $S^*$-equivalent 
if and only if $[F] =[F']$ as elements in $H_2(\Si \times I,L;\ZZ/2)$ (see \cite[Lemma 1.5 \& Proposition 1.6]{Boden-Chrisman-Karimi-2021}). 

Thus, for a non-split checkerboard colorable link $L$ in any thickened surface, there are two $S^*$-equivalence classes of spanning surfaces. Indeed, the black and white surfaces represent the two $S^*$-equivalence classes, and any other spanning surface for $L$ is $S^*$-equivalent to either the black or the white surface.

\begin{figure}[ht]
\centering
\includegraphics[scale=0.90]{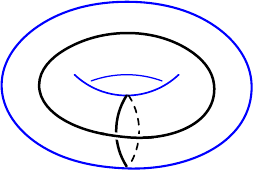} 
\vspace{-1mm}
\caption{\small  A non-checkerboard colorable link.}\label{figure-5}
\end{figure}

More generally, given a spanning surface $F$ for a link $L \subset \Si \times I$ in a connected thickened surface, we can construct a new spanning surface  by connecting $F$ to a parallel copy of $\Si$ near $\Si \times \{0\}$ by a small thin tube $\tau$. Let $F' = F \#_{\tau} \Si$ denote the new spanning surface. When $L$ is non-split, it is not difficult to see that $F$ and $F'$ represent the two $S^*$-equivalence classes of spanning surfaces.

In the case that $F$ is a checkerboard surface for $L$, the correction term $\frac{1}{2}e(F,L)$ is given by a sum of incidence numbers of crossings of type I or II (see \Cref{figure-3}). The incidence number of a crossing $x$ is denoted $\eta_x$ and is defined by setting 
$$\eta_x=\begin{cases} 1 & \text{if $x$ is type $a$,}  \\
-1 & \text{if $x$ is  type $b$.} 
\end{cases}$$

Specifically, if $B$ and $W$ denote the black and white surfaces of the checkerboard coloring, set
$$ \mu_W(D) = \sum_{x \; \text{type I}} -\eta_x\qquad \text{ and } \qquad \mu_B(D)
=\sum_{x \; \text{type II}} \eta_x.$$
By Lemma 2.4 \cite{Boden-Chrisman-Karimi-2021}, we  see that
$\mu_W(D)=-\frac{1}{2}e(W,L)$ and $\mu_B(D) = -\frac{1}{2}e(B,L).$

\setcounter{section}{3} \noindent
\subsection{Definite surfaces} \label{S3} 
In this section, we show that a connected checkerboard colorable link diagram $D$ on $\Si$ is alternating 
if and only if its checkerboard surfaces are definite and of opposite sign. 

\begin{definition}
A compact, connected surface $F$ in $\Si\times I$ is positive (or negative) definite if its Gordon-Litherland pairing $\cG_{F}$ is.
\end{definition}

Suppose $D$ is a connected link diagram on $\Si$ such that 
\begin{itemize}
\item[{(i)}] $\Si\sm D$ is  a union of disks, 
\item[{(ii)}] $D$ is checkerboard colorable. 
\end{itemize}
Choose a checkerboard coloring of $D$ and let $B,W$ be the black and white surfaces.
Then $B,W$ have first Betti numbers
\begin{equation} \label{eqn:firstbetti}
\begin{split} 
b_{1}(W) &= 2g+\be-1, \\
b_{1}(B) &= 2g+\al-1, 
\end{split}
\end{equation}
where $g=\text{genus}(\Si)$ is the genus of $\Si$,
$\al$ is the number of white disks and $\be$ is the number of black disks. 

The Euler characteristic of $\Si$ satisfies $\chi(\Si)=2-2g=c(D)-2c(D)+(\al+\be),$ 
where $c(D)$ denotes the number of crossings of $D$. Thus we have 
\begin{equation}\label{eqn:star}
\al+\be=2-2g+c(D).
\end{equation} 
In particular, combining equations \eqref{eqn:firstbetti} and \eqref{eqn:star} gives that
\begin{equation} \label{eqn:bettisum}
b_1(W)+b_1(B) = 2g +c(D).
\end{equation}

\begin{remark} \label{kamada}
According to \cite[Lemma 7]{Kamada-02}, any link diagram $D$ on a surface which is alternating and cellularly embedded is checkerboard colorable.  
\end{remark}

\begin{lemma}\label{BWdefinite}
If $D$ is a link diagram on $\Si$ which is cellularly embedded and alternating, then the black and white surfaces are definite and of opposite sign.
\end{lemma}

\begin{proof}
By \Cref{kamada}, we see that $D \subset \Si$ is checkerboard colorable, and we choose the coloring so that every crossing has type $b$. The black surface $B$ is a union of  disks and half-twisted bands, and with this choice each band has a \emph{left-handed} half-twist. (The white surface $W$ is likewise a union of disks and half-twisted bands, and each band has a right-handed half-twist.) 

Let $\Ga_{B}$ be the Tait graph for the black surface; it is a graph in $\Si$ with one vertex for each black disk and one edge for each crossing. Choose arbitrarily an orientation for the edges, and label the edge associated with a crossing $x$ of $D$ with its incidence number $\eta_x$. (Since every crossing has type $b$, each edge of $\Ga_B$ is labeled $-1$.) The  black surface $B$ admits a deformation retraction onto $\Ga_{B}$, hence $H_1(B)\cong H_1(\Ga_{B})$. 



Now consider an element $\ga$ in $H_1(B) \cong H_1(\Ga_B),$ which we can view as a sum $\ga_1 + \cdots  + \ga_k$ of cycles in $\Ga_B$. We can write each cycle $\ga_i$ as a sum of oriented edges. If the same edge occurs with opposite signs in two of the cycles, then they can be eliminated without altering the homology class of $\ga$.
Thus, we can assume that no edge occurs with opposite signs in two cycles.

Let $a_i$ be a simple closed curve in $B$ corresponding to $\ga_i$ for $i=1,\ldots, k$. We can assume that the curves $a_i$ and $a_j$ intersect transversely for $i \neq j$, and that their intersection points lie within the black regions. This implies that $a_i \cap a_j$ consists of finitely many points, none of which contribute to  $G_F(a_i,a_j)$ (see \Cref{figure-local} (right)). Since no edge occurs with opposite signs in the cycles $\ga_i$ and $\ga_j$, whenever two curves $a_i$ and $a_j$ cross the same band, they cross it in the same direction.

Since $D$ is cellularly embedded, any nontrivial cycle $\ga_i$ will have edge set with $\ell_i>0$ edges. Therefore, $\lb \ga_i \rb =\cG_B(\ga_i,\ga_i)=-\ell_i<0$. This step follows by computing it as a sum of local contributions, 
each of which comes from a band crossing of $a_i$ and is negative. For $i \neq j$, we see that $\cG_B(\ga_i,\ga_j) \leq 0$. This follows by again viewing it as a sum of local contributions, each of which comes from a band crossing of $a_i$ and $a_j$ and is non-positive. (This step uses the condition that whenever the two curves $a_i$ and $a_j$ cross the same band, they cross it in the same direction.)
Thus, for $\ga = \ga_1+\cdots + \ga_k,$ we see that
\begin{equation*}
\begin{split}
\lb \ga \rb_B = \cG_B(\ga, \ga) &= 
\cG_B(\ga_1 +\cdots+\ga_k, \ga_1 +\cdots+\ga_k), \\
& = \sum_{i=1}^k \cG_B(\ga_i,\ga_i) +2 \sum_{i<j}\cG_B(\ga_i,\ga_j) <0.
\end{split}
\end{equation*}
Since this holds for all nontrivial homology classes in $H_1(B)$, it follows that $B$ is negative definite.

A similar argument shows that the white surface $W$ is positive definite.
\end{proof}

\begin{remark}\label{nonsingular}
If $F$ is a spanning surface for $L$ such that $\cG_F$ non-singular, then the Gordon-Litherland pairing will be non-singular for any spanning surface $S^*$-equivalent to $F$, see  
\cite[Theorem 2.5]{Boden-Chrisman-Karimi-2021}.
\end{remark}

The next result is a restatement of \cite[Theorem 3.1]{Boden-Chrisman-Karimi-2021}.

\begin{theorem}\label{minimal-genus}
Suppose $L \subset \Si \times I$ is a link with a connected, checkerboard colored diagram $D \subset \Si$. Let $B,W$ be the two spanning surfaces associated to the black and white regions, respectively.
If the Gordon-Litherland pairings $\cG_B$ and $\cG_W$ are both non-singular, then $L \subset \Si \times I$ is non-split and has minimal genus. In particular, this implies that $D$ is cellularly embedded.
\end{theorem}

\begin{proof}
The proof is by contradiction, and here we sketch the main idea and refer
to the proof of \cite[Theorem 3.1]{Boden-Chrisman-Karimi-2021} for further details.

The key step is to show that if $D \subset \Si$ is not cellularly embedded, then one of $\cG_B$ or $\cG_W$ is singular. 

Assume then that $D$ is not cellularly embedded. Then we can find  a non-contractible simple closed curve $\ga$  disjoint from $D$. Since $\ga$ and $D$ are disjoint, $\ga$ must be contained entirely in either one of the black regions or one of the white regions.   
Without loss of generality, we can assume that $\ga$ lies in a black region.

We claim there exists a simple closed curve $\al$ in $\Si$ lying entirely in a black region such that its homology class $[\al] \in H_1(\Si)$ is nontrivial. Indeed, if $\ga$ is non-separating, then we can take $\al =\ga.$ Otherwise, if $\ga$ is separating, then since $D$ is connected, it must lie in one of the connected components of $\Si \sm \ga$. Both components have positive genus (since $\ga$ is non-contractible), and the component disjoint from $D$ is contained entirely in a black region. There exists  a simple closed curve $\al$ in that component with $[\al]\neq 0$ in $H_1(\Si)$.
 
Consider now the map $H_1(B) \to H_1(\Si)$ induced by  
$B \hookrightarrow \Si \times I \to \Si$, the composition of inclusion and projection.
Since $[\al]$ is nontrivial in $H_1(\Si)$, it must also be nontrivial in $H_1(B).$
Further, since $\al$ is a simple closed curve, the set 
$\{[\al]\}$ can be extended to a basis $U$ for $H_1(B)$. 
Since $\al$ lies entirely within one of the black regions, we have $\cG_{B}(\al, \al)=\lb \al^{}_{} \rb_B=0$. Any other element of $U$ can be represented as a simple closed curve $\be$ on $B$. Since $\al$ lies entirely in a black region, the two curves $\al$ and $\be$ have only intersection points; there are no points where $\al$ crosses over or under $\be$. However, as we have seen, an intersection point contributes 0 to the pairing, and thus it follows that  $\cG_{B}(\al, \be)=0$. This holds for all $\be \in U$, therefore the Gordon-Litherland pairing $\cG_{B}$ is singular.

We now prove that $L$ is non-split. Suppose to the contrary that $L$ is split, and let $D' \subset \Si$ be a disconnected diagram for $L$. Notice that $D'$ is not cellularly embedded, and that the checkerboard surfaces need not be connected. By adding small tubes, we can connect them. However, one or both of the resulting spanning surfaces will have singular Gordon-Litherland pairing. This is a contradiction (cf. \Cref{nonsingular}).
\end{proof}

\begin{corollary} \label{cor-min-g}
Any link $L$ in $\Si \times I$ represented by a cellularly embedded alternating diagram is non-split and has minimal genus.
\end{corollary}

\begin{proof}
Let $D$ be a cellularly embedded alternating diagram for $L$. Then $D$ is checkerboard colorable (\Cref{kamada}), and \Cref{BWdefinite} implies the black and white surfaces are definite. In particular, their Gordon-Litherland pairings are non-singular. The conclusion now follows from \Cref{minimal-genus}. \end{proof}

By convention, given an alternating diagram for a link $L$ in $\Si \times I$, we will choose the coloring in which every crossing has type $b$. With this choice, the white surface becomes positive definite and the black surface becomes negative definite.

\begin{lemma}
Suppose $D$ is a connected alternating diagram for a link $L$ in $\Si\times I$ with checkerboard coloring such that every crossing has type $b$. Then 
$$\si_{W}(L)-\si_{B}(L)=2g(\Si).$$ 
\end{lemma}

\begin{proof}
In general we have 
\begin{eqnarray*}
\si_{W}(L) &=& \sig(\cG_W)-\mu_W(D), \\
\si_{B}(L) &=& \sig(\cG_B) -\mu_B(D). 
\end{eqnarray*}

Since all crossings have type $b$ and referring to \Cref{figure-7}, we see that 
\begin{eqnarray*}
\mu_W(D)&=& \sum_{x \; \text{type I}}-\eta_x = \sum_{x \; \text{type I}}\ep_x=
c_{+}(D), \\
\mu_B(D)&=& \sum_{x \; \text{type II}}\eta_x= \sum_{x \; \text{type II}}\ep_x=
-c_{-}(D),
\end{eqnarray*} 
where $c_+(D)$ is the number of positive crossings of $D$
and $c_-(D)$ is the number of negative crossings.
Hence 
\begin{equation} \label{eqn:correction}
\mu_W(D) -\mu_B(D) = c_{+}(D) + c_{-}(D) = c(D).
\end{equation}

\begin{figure}[ht]
\centering
\includegraphics[scale=0.90]{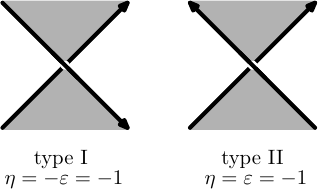} 
\vspace{-3mm}
\caption{\small Two type $b$ crossings.}\label{figure-7}
\end{figure}

\Cref{BWdefinite} shows that $W$ is positive definite and $B$ is negative definite, hence

\begin{eqnarray*}
\si_{W}(L)-\si_{B}(L) &=& \sig(\cG_W)-\mu_W(D)-(\sig(\cG_B)-\mu_B(D)), \\
 &=& b_1(W)+b_1(B)- (\mu_W(D) -\mu_B(D)) =2g,
\end{eqnarray*}
where the last step follows from equations \eqref{eqn:bettisum} and \eqref{eqn:correction}.
This completes the proof.
\end{proof}

\begin{proposition} \label{prop-BW}
Let $D\subset \Si$ be a cellularly embedded, checkerboard colorable link diagram. Then $D$ is alternating if and only if the black and white surfaces are definite and of opposite sign.
\end{proposition}

\begin{proof}
If $D$ is alternating, then \Cref{BWdefinite} gives the desired conclusion.

Conversely, suppose $B$ is negative definite and $W$ is positive definite. Let $a_{\pm}$ be the number of type $a$ crossings of $D$ with $\ep_x=\pm 1$, and $b_{\pm}$ be the number of type $b$ crossings of $D$ with $\ep_x=\pm 1$. Then  
 $$
\mu_{W}(D) = a_{-}-b_{+} \qquad \text{ and } \qquad
\mu_{B}(D) = -a_{+}+b_{-}.$$
It follows that
\begin{eqnarray*}
\mu_{W}(D)-\mu_{B}(D) &=& a_{-} - b_{+}-(-a_{+}+b_{-}), \\
&=&a(D)-b(D),
\end{eqnarray*}
where  $a(D)=a_{+}+a_{-}$ is the total number of type $a$ crossings and $b(D)=b_{+}+b_{-}$ is the total number of type $b$ crossings.

Therefore, 
\begin{equation}\label{eqn:dagger}
|\mu_{W}(D)-\mu_{B}(D)|\leq c(D),
\end{equation}
with equality if and only if $a(D)=0$ or $b(D)=0$. In the first case, all crossings have type $b$, and in the second, they all have type $a$. In either case, we see that $D$ is alternating. Thus, it only remains to show that $|\mu_{W}(D)-\mu_{B}(D)|= c(D)$.

Given a spanning surface $F$ for $L$, by the discussion after \Cref{GL-formula}, we obtain a new surface $F\#_\tau \Si$ by connecting it to a parallel copy of $\Si$ by a thin tube $\tau$. 
The surfaces $F$ and $F\#_\tau \Si$ are not $S^*$-equivalent.

Since $B$ and $W$ are not $S^*$-equivalent, 
and since a non-split link in $\Si \times I$ has exactly two $S^*$-equivalence classes of spanning surfaces, it follows that $B$ and $W\#_\tau \Si$ must be $S^{*}$-equivalent. By \Cref{GL-formula}, $\si_{B}(L)=\si_{W\#_\tau \Si}(L)$. We have $|\si_{W}(L)-\si_{W\#_\tau \Si}(L)|\leq 2g$.
Hence $|\si_{W}(L)-\si_{B}(L)|\leq 2g$.  
Further, since $B$ and $W$ are definite surfaces of opposite sign,
we see that $|\sig(\cG_W)-\sig(\cG_B)| = b_{1}(W)+b_{1}(B).$
These two observations, combined with \cref{eqn:bettisum}
and inequality \eqref{eqn:dagger}, show that:
\begin{eqnarray*}
2g &\geq & |\si_{W}(L)-\si_{B}(L)|, \\
  &=& |\sig(\cG_W)-\mu_{W}(D)-(\sig(\cG_B)-\mu_{B}(D))|,\\
 &\geq & |\sig(\cG_W)-\sig(\cG_B)|-|\mu_{W}(D)-\mu_{B}(D)|,\\
 &=& b_{1}(W)+b_{1}(B)-|\mu_{W}(D)-\mu_{B}(D)|,\\
 &=& 2g+c(D)-|\mu_{W}(D)-\mu_{B}(D)| \;\; \geq \;\; 2g. 
\end{eqnarray*}  
Therefore we must have equality throughout, and it follows that $D$ is alternating.  
\end{proof}

\setcounter{section}{4} \noindent
\subsection{Characterization of alternating links in thickened surfaces} \label{S4} 
In this section, we establish our main result,
\Cref{thm:main2}. It is a combination of \Cref{BWdefinite} and \Cref{prop-BW}.

If $L$ is a link in the thickened surface $\Si \times I,$ let $\nu(L)$ be a tubular neighborhood of $L$ and  let $X_L=\Si\times I \sm \Int(\nu(L))$ denote the exterior of $L$.
The next result is a restatement of part of Proposition 6.3 from \cite{Carter-Silver-Williams-2014}.
Recall that a link $L$ in a 3-manifold $M$ is said to be \textit{local} if it is contained in an embedded 3-ball $B$ in $M.$

\begin{proposition}[Carter-Silver-Williams]\label{prop-CSW}
If $\Si$ is a surface of genus $g\geq 1$ and $L$ is a non-split and non-local link in $\Si\times I$, then the exterior $X_L$ is irreducible.
\end{proposition}

\begin{proof} A detailed proof can be found in \cite{Carter-Silver-Williams-2014}, so we only  sketch the argument. Since $\Si$ has genus $g \geq 1$, the universal cover of $\Si\times I$ is
$\RR^2\times I$, which is irreducible, then by \cite[Proposition 1.6]{Hatcher-3manifolds} it follows that $\Si \times I$ is irreducible.
Any embedded 2-sphere in $X_L \subset \Si \times I$ must bound a 3-ball $Y \subset \Si \times I$, and the hypotheses ensure that $Y\subset X_L$. This completes the argument.
\end{proof}

The next result is an analogue of Lemma 3.1 from \cite{Greene}. 

\begin{lemma}[Greene] \label{lemma3-1}
If $F\subset \Si \times I$ is a definite surface with $\partial F =L,$ then $b_1(F)$ is
minimal over all spanning surfaces for $L$ which are $S^*$-equivalent to $F$ and have the same Euler number as $F$. If $F'$ is another such surface with $b_1(F')=b_1(F),$ then $F'$ is definite and of the same sign as $F$.
\end{lemma}
\begin{proof}
If $F'$ is $S^*$-equivalent to $F$, then \Cref{GL-formula} implies that $\si_F(L)=\si_{F'}(L).$ If, in addition, $e(F)=e(F')$, then it follows that $\sig(\cG_F)= \sig(\cG_F').$ 

Now suppose $F$ is definite. Then we have
$$b_1(F) = |\sig(\cG_F)| = |\sig(\cG_{F'})| \leq b_1(F'),$$
which shows the first claim. 

If, in addition, $b_1(F')=b_1(F),$ then we have $b_1(F') = |\sig(\cG_{F'})|,$ hence $F'$ must also be definite. Since $\sig(\cG_{F}) = \sig(\cG_{F'}),$ it follows that $F$ and $F'$ must have the same sign.
\end{proof}

\begin{corollary} If $F \subset \Si \times I$ is definite, then it is incompressible.
\end{corollary}
\begin{proof} Suppose to the contrary that $F$ is compressible. Let $F'$ be the surface obtained from $F$ by a compression. Then $F'$ is $S^*$-equivalent to $F$ and $b_1(F') < b_1(F).$ Further, $e(F') =e(F)$. However, this is impossible, for it would contradict \Cref{lemma3-1} if $F$ is definite.
\end{proof}

\begin{lemma}\label{lemmaA}
Let $S \subset \Si \times I$ be a compact, closed, connected, oriented surface with
$[S]\neq 0$ in $H_2(\Si\times I)$. Then $g(S)\geq g(\Si)$. If $g(S)= g(\Si)$,
then $S$ is incompressible in $\Si \times I$.
\end{lemma}
\begin{proof}
Any surface $S$ with genus $g(S)< g(\Si)$ has $[S]= 0$ in $H_2(\Si\times I)$. This proves the first statement. 

Assume now that $g(S)= g(\Si)$. If $S$ is not incompressible, there exists a non-contractible simple closed curve $\ga$ in $S$ which bounds a disk in $\Si \times I$. Let $S'$ be the surface obtained by cutting $S$ along $\ga$ and gluing in two disks along the newly created boundary components. Then $[S'] =[S]$ in $H_2(\Si \times I).$ If $\ga$ is non-separating on $S$, then $g(S') = g(S)-1 < g(\Si)$ and $[S']=0$, a contradiction. Otherwise, if $\ga$ is separating on $S$, then $S'$ is a disjoint union $S_1' \sqcup S'_2$ of two surfaces of positive genus satisfying $g(S')=g(S_1') + g(S'_2) =g(S)$. Hence $0<g(S_i')< g(S)$ for $i=1,2$. It follows that $[S'_1] =0$ and $[S_2']=0$ in $H_2(\Si \times I).$ 
Thus $[S'] = [S_1'] + [S_2'] = 0,$ again a contradiction. This proves the second statement.
\end{proof}

\begin{remark} \label{Cor-to-Waldhausen}
In the above lemma, if $g(S)= g(\Si) \geq 1$, then Corollary 3.2 of \cite{Waldhausen-1968}
applies to show that $S$ is isotopic to $\Si \times \{t_0\}$ for $0<t_0<1.$  
\end{remark}

The next result is Lemma 3.3 from \cite{Greene}. 
The proof  is the same as in \cite{Greene} so we will not repeat it here.

\begin{lemma}[Greene] \label{lemma3-3}
If $S$ is definite and $S' \subset S$ is a compact subsurface with connected boundary, then $S'$ is definite.
\end{lemma}

We can now prove the following analogue of Lemma 3.4 of \cite{Greene} for links in thickened surfaces. 

\begin{lemma}\label{lemma3-4}
Let $L$ be a link in $\Si\times I$, where $\Si$ has genus $g\geq 1$. Suppose further that
$L$ has a positive definite spanning surfaces $P$ and a negative definite spanning surface $N$.
If $P$ and $N$ intersect transversely in $X_L$ such that the number of components of $P \cap N \cap X_L$ is minimized (up to isotopy), then $P\cap N\cap X_L$ does not contain a simple closed curve.
\end{lemma}

\begin{proof}
Suppose $P$ and $N$ are positive and negative definite surfaces, respectively, and $\ga$ is a simple closed curve contained in $P\cap N\cap X_L$.
Let $\nu(\ga)$ be a small regular neighborhood of $\ga$ in $\Si\times I$, which contains no other intersection of $P$ and $N$. Clearly $\nu(\ga)$ is a $D^2$ bundle over $\ga$, and since $\Si\times I$ is orientable, so is $\nu(\ga)$ and it follows that $\nu(\ga)\approx S^1\times D^2$ is a trivial bundle.
The annuli $P \cap \nu(\ga)$ and $N \cap \nu(\ga)$ intersect only in $\ga$, so the framing $\lb\ga \rb_{P}$ of $\ga$ in $P$ is equal to the framing $\lb\ga \rb_{N}$ of $\ga$ in $N$.  Since $P$ and $N$ are positive and negative definite, respectively, we have  
$$\lb\ga \rb_{N}\leq 0 \leq \lb\ga \rb_{P}.$$
Therefore, $\lb\ga \rb_{P}=\lb\ga \rb_{N}=0$. It follows that $\ga$ is null-homologous in both $P$ and $N$. Thus, $\ga$ is separating on $P$ and $N$. Let $P'$ and $N'$ be orientable subsurfaces of $P$ and $N$, respectively, with $\partial P'=\partial N'=\ga$. By choosing $\ga$ an innermost curve on $N$, we can arrange that $N'$ is disjoint from $P.$ By \Cref{lemma3-3}, $P'$ is positive definite and $N'$ is negative definite. 
 
Set $V = P' \cup_ga N'$. Since $P'$ and $N'$ are both orientable, $V=P' \cup_\ga N'$ is a closed orientable surface,
and $\frac{1}{2}e(P',\ga)=0=\frac{1}{2}e(N',\ga)$.
Therefore, $\si_{P'}(\ga) = \sig(\cG_{P'})$ and $\si_{N'}(\ga)=\sig(\cG_{N'})$. 

Suppose firstly that $P'$ and $N'$ are $S^*$-equivalent.
Then 
$$0\leq b_{1}(P') = \sig(\cG_{P'}) = \si_{P'}(\ga) = \si_{N'}(\ga) = \sig(\cG_{N'})=-b_{1}(N')\leq 0.$$
Therefore, $P'$ and $N'$ are disks in this case.

Suppose now that $P'$ and $N'$ are not $S^*$-equivalent.
Then $P'$ and $N'\#_\tau \Si$ are $S^*$-equivalent, and so
$\sig(\cG_{P'}) = \sig(\cG_{N' \#_\tau \Si})$.
Therefore, if $g=\text{genus}(\Si),$ we see that
$$2g \geq \sig(\cG_{N' \#_\tau \Si}) - \sig(\cG_{N'}) = \sig(\cG_{P'}) - \sig(\cG_{N'}) = b_1(P') + b_1(N').$$ 
In particular, $V$ is a closed, orientable surface with $\text{genus}(V) \leq g$.
In fact, its genus must be exactly $g$, since otherwise we would have $[V] = 0$ in  $H_2(\Si;\ZZ_2)$, which would  imply that $P'$ and $N'$ are $S^*$-equivalent.

\Cref{lemmaA} applies and shows that $V$ is incompressible in $\Si \times I.$ Therefore, by
\cite[Corollary 3.2]{Waldhausen-1968}, it follows that $V$ is isotopic to $\Si \times \{t_0\}$ for some $0<t_0<1$ (see \Cref{Cor-to-Waldhausen}). Since $g(V)\geq 1$, there is a non-separating simple closed curve $\al$ in either $P'$ or $N'$. Assume $\al\subset P'$. Then $[\al]\neq 0 \in H_1(P')$, so $\lb \al^{}_{} \rb_{P'}>0$. On the other hand, since $\al \subset P'\subset V$ and $V$ is isotopic to $\Si \times \{t_0\}$, the framing of $\al$ in $P'$ is the same as the framing of $\al$ in $V$, so $\lb \al^{}_{} \rb_{P'}=0$. This is a contradiction, therefore, $P'$ and $N'$ are $S^*$-equivalent. It follows that $P'$ and $N'$ are both disks, and $V$ is a 2-sphere. By \Cref{prop-CSW}, $X_L$ is irreducible, thus, $V$ bounds a 3-ball. Using this ball, we can set-up an isotopy which separates $P'$ and $N'$, and makes the number of components of $P \cap N \cap X_L$ smaller. This is a contradiction, and the result follows.
\end{proof}

\begin{theorem} \label{thm:main2} 
Suppose $L \subset \Si \times I$ is a non-split link with minimal genus.  
Then $L$ is alternating if and only if there exist positive and negative definite spanning surfaces for $L$.
\end{theorem} 

\begin{proof}
For $g(\Si)=0$, the result follows from \cite{Greene}, therefore we assume $g(\Si)\geq 1$. 

Suppose $L$ is non-split and $D$ is an alternating diagram for $L$. Since $L$ has minimal genus,  $D$ is cellularly embedded. Further, $D$ is checkerboard colorable by \Cref{kamada}, and \Cref{BWdefinite} implies that the checkerboard surfaces $W$ and $B$ are positive and negative definite, respectively. This proves one direction, and it remains to prove the other. 

Suppose then that $P$ and $N$ are two definite spanning surfaces for $L$, with $P$ positive definite and $N$ negative definite.

Let $X_L=\Si\times I\sm \Int(\nu(L))$ be the exterior of  $L$. We write $\partial X_L = \partial_1 X_L \cup \cdots \cup \partial_m X_L$ according to the components of the link $L = K_1 \cup \cdots \cup K_m$. Clearly, each $\partial_i X_L$ is a torus. Assume further that,  $P$ and $N$ intersect transversely in $X_L$ such that the number of components of $P \cap N \cap X_L$ is minimized.

For $i=1,\ldots, m$ set $\la^P_i=P\cap \partial_i X_L$ and $\la^N_i=N\cap \partial_i X_L$. Thus $\la^P_i$ and $\la^N_i$ intersect transversely in $\partial_i X_L$. We further set
$\la^P = \bigcup_i \la^P_i$ and $\la^N = \bigcup_i \la^N_i.$

By \Cref{lemma3-4}, we can assume that $P\cap N\cap X_L$ does not contain any closed components.  Thus $P\cap N\cap X_L$ is a union of arc components which we call \textit{double arcs}. Each double arc connects a pair of distinct points in $\la^P\cap \la^N$. Thus $\la^P\cap \la^N$ consists of an even number of points, equal to twice the number of double arcs. Since $P$ is positive definite and $N$ is negative definite, the number of points in $\la^P_i\cap \la^N_i$ is equal to the difference in framings $\tfrac{1}{2}\lb [K_i]\rb_{P}-\tfrac{1}{2}\lb [K_i]\rb_{N}$. Summing over the components, we get that
$$\sum_{i=1}^m \left( \tfrac{1}{2}\lb [K_i]\rb_{P}-\tfrac{1}{2} \lb [K_i]\rb_{N} \right) = e(N)-e(P).$$
The number of arc components in
$P \cap N \cap X_L$ is therefore equal to $\frac{1}{2}(e(N)-e(P))$. 

An orientation on $X_L$ induces one on $\partial X_L$, and an orientation on each $K_i$ induces ones on $\la^P_i$ and $\la^N_i$. This defines a sign $\ep_x \in \{\pm 1\}$ for each point $x\in \la^P_i\cap \la^N_i$. The sign $\ep_x$ is positive if the orientation of $\la^P_i$ followed by the orientation of $\la^N_i$ at $x$ agree with the orientation of $\partial_i X_L$, and it is negative otherwise. A key point is that every point in $\la^P_i\cap \la^N_i$ has the same sign; this follows from the fact that 
$$\#(\la^P_i\cap \la^N_i)=\tfrac{1}{2}\lb [K_i]\rb_{P}-\tfrac{1}{2} \lb [K_i]\rb_{N}=\left|{\sum \ep_{x}}\right|,$$
where the sum on the right is taken over all $x \in \la^P_i\cap \la^N_i.$

As explained in \S 2 of \cite{Howie}, there are two kinds of double arcs; one is called a \textit{standard} double arc and the other is called a \textit{parallel} double arc. A double arc of $P \cap N$ with endpoints $x$ and $y$ is standard if $\ep_x = \ep_y$, and it is parallel if $\ep_x = -\ep_y.$ By the previous observations, we see that every double arc is standard. Each double arc extends to give an arc $a$ in $P\cap N$ with $\partial a=a\cap L$.
Since the double arc is standard, there is a neighborhood $V$ of $a$ modelled on the intersecting checkerboard surfaces in a standard neighborhood of a crossing in a link diagram. See Figures \ref{figure-8} and \ref{figure-9}. 

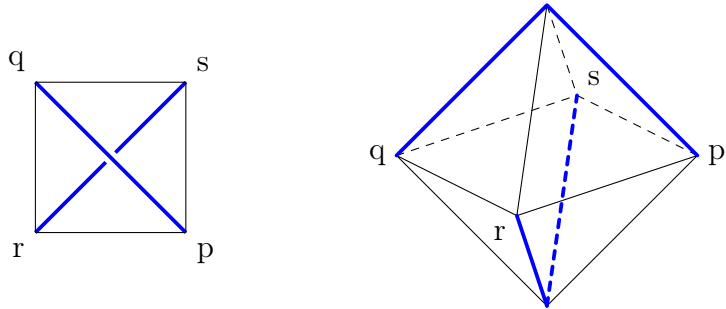
\begin{figure}
\centering 
\begin{tikzpicture}[baseline]
 \draw [line width=0.5mm, blue] (1,1) -- (3,3);
 \draw [white, line width=1ex] (3,1) -- (1,3);
 \draw [line width=0.5mm, blue] (3,1) -- (1,3);
 \draw [thin] (1,1) -- (1,3)node[above left] {q} --(3,3)node[above right] {s} --(3,1)node[below right] {p} --(1,1)node[below left] {r} ;
 \end{tikzpicture}
\qquad \qquad
 \vspace{1cm}
 \begin{tikzpicture}[z={(-.2cm,-.4cm)}, 
                    line join=round, line cap=round 
                   ]
  \draw (0,2,0) -- (-2,0,0) node[left] {q} -- (0,-2,0)  -- (2,0,0) node[right] {p}-- (0,2,0) -- (0,0,2) -- (0,-2,0) (2,0,0) -- (0,0,2)node[below left] {r} -- (-2,0,0);
  \draw[dashed] (0,2,0) -- (0,0,-2) (2,0,0) -- (0,0,-2) node[above right] {s}-- (-2,0,0);
\draw[line width=0.5mm, color=blue] (-2,0,0) --( 0,2,0) --(2,0,0);
\draw[line width=0.5mm, color=blue] (0,0,2) --( 0,-2,0);
\draw[line width=0.5mm, dashed,color=blue] (0,-2,0) -- (0,0,-2);
\end{tikzpicture} 
\vspace{-12mm}
\caption{\small  A standard crossing (left) placed on an octahedron (right).}\label{figure-8}
\end{figure}

\begin{figure}[ht]
\centering
\includegraphics[scale=0.85]{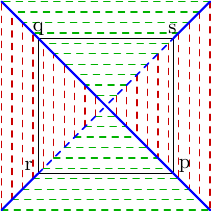}
\hspace{.4cm} 
\includegraphics[scale=0.85]{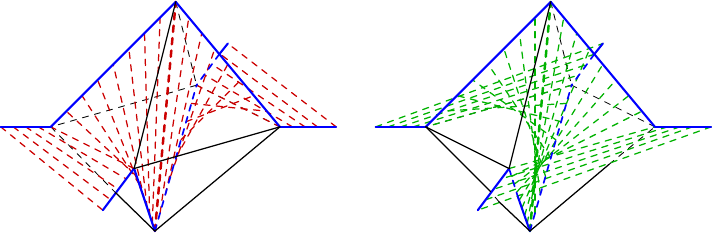}
\vspace{-2mm}
\caption{\small  The two ruled surfaces near a crossing. The heavily dashed vertical lines depict the standard arc of intersection of the two surfaces.}\label{figure-9}
\end{figure}

We can choose local coordinates near the crossings so that the arcs of the link lie in the $xy$ plane except at the crossings. 
At each crossing we place an octahedron  which intersects the $xy$ plane in the square with vertices $p=(1,-1,0)$, $q=(-1,1,0)$, $r=(-1,-1,0)$,  and $s=(1,1,0)$. We assume that the over-crossing arc connecting $p$  to $q$ is given by
$$\beta(t)=(1-2t,-1+2t,\text{min}(2t,2-2t))\ \ \text{for}\ \ 0\leq t\leq 1,$$
and the under-crossing arc connecting $r$ to $s$ is given by
$$\ga(t)=(-1+2t,-1+2t,\text{max}(-2t,-2+2t))\ \ \text{for}\ \ 0\leq t\leq 1.$$ 

In the standard crossing, the black surface is parametrized by
$$B(s,t)=\beta(t)(1-s)+s\cdot\ga(t)\ \ \text{for}\ \ 0\leq s,t\leq 1,$$
and the white surface is parametrized by  
$$W(s,t)=\beta(t)(1-s)+s\cdot\ga(1-t)\ \ \text{for}\ \ 0\leq s,t\leq 1.$$
Notice that, at the crossing, the black surface contains a left half-twist, whereas the white surface contains a right half-twist. The black and white surfaces intersect in the vertical arc $(0,0,1-2t),$ for $0\leq t \leq 1,$ which connects $(0,0,1)$ to $(0,0,-1)$. 

Thus, any standard arc $a$ has a neighborhood $V\subset \Si\times I$ such that $a$ is \textit{vertical} and the projection $p\colon \Si\times I\to \Si$ maps $(P\cup N\sm a)\cap V$ homeomorphically onto a once-punctured disk in $\Si$. 

Let $A$ denote the union of all the double arcs in $P\cap N\cap X_L$. \Cref{lemma3-4} implies that $P\cap N\cap X_L$ does not contain any simple closed curves. Thus it follows that $P\cup N\sm A$ is a two-dimensional manifold. Furthermore, collapsing the standard models of each of the double arcs $a$ in $A$ down, we see that $\nu(P\cup N)$ is homeomorphic to $\nu(S)$ for some connected surface $S$ embedded in $\Si\times I$. We can identify $\nu(S) \approx S\times I$ in such a way that the double arcs are all mapped to distinct points under projection $S\times I\to S$. 

Set $c=\frac{1}{2}(e(N)-e(P))$, which is equal to the number of arc components in $P\cap N \cap X_L.$
Therefore, $P \cap N = L \cup A$ and has Euler characteristic $-c.$

\bigskip
\noindent
{\textbf{Claim:}} $\chi(S)=\chi(\Si)$.

\bigskip \noindent
{\textit{Proof of the Claim.}} Since $\nu(S)=\nu(P\cup N)$, we have that
\begin{equation}  \label{claim}
\begin{split}
\chi(S)&=\chi(\nu(S))=\chi(\nu(P\cup N))=\chi(P\cup N), \\
&=\chi(P)+\chi(N)-\chi(P\cap N), \\ 
&=(1-b_{1}(P))+ (1-b_{1}(N))+c, \\
&=2-(b_1(P)+b_{1}(N)-c). 
\end{split}
\end{equation} 
On the other hand, computing the signature of $L$ using $P$ and $N$, we see that
\begin{equation*}
\begin{split}
\si_{P}(L)&=\sig(\cG_P)+\tfrac{1}{2}e(P, L),\\
\si_{N}(L)&=\sig(\cG_N)+\tfrac{1}{2}e(N, L).
\end{split} 
\end{equation*}
Thus, by \cref{eqn:Euler}, we have
\begin{equation} \label{sigma}
\begin{split}
\si_{P}(L)-\si_N(L)&=\sig(\cG_P)-\sig(\cG_N)+\tfrac{1}{2}(e(P,L)-e(N,L)),\\
&=\sig(\cG_P)-\sig(\cG_N)+\tfrac{1}{2}(e(P)-e(N)),\\
&=b_1(P)+b_{1}(N)-c.  
\end{split}
\end{equation}
Substituting \cref{sigma} into \cref{claim}, we see that
\begin{equation} \label{e-chi}
\chi(S)=2-(\si_{P}(L)-\si_N(L)).
\end{equation}

We will now show that $P$ and $N$ are not $S^*$-equivalent. Suppose to the contrary that $P$ and $N$ are $S^*$-equivalent. Then $\si_{P}(L)=\si_N(L)$, and the equations above give that $\chi(S)=2$. But since $\Si \times I$ is irreducible, any 2-sphere bounds a 3-ball. Therefore, the link $L$, which lies on $S$, must be contained in a 3-ball. This however contradicts our assumption that $L \subset \Si \times I$ has minimal genus $g(\Si)\geq 1.$
Thus, $P$ and $N$ cannot be $S^*$-equivalent.

Since $N$ is not $S^*$-equivalent to $P$, it follows that $N$ must be $S^*$-equivalent to $P\#_{\tau}\Si$. Thus
$$| \si_{P}(L)-\si_N(L)| =|\si_{P}(L)-\si_{P\#_{\tau}\Si}(L)| \leq 2g(\Si).$$
By \cref{e-chi}, we see that 
$$\chi(S) \geq 2 - 2g(\Si) = \chi(\Si).$$ 
If this inequality were strict, i.e., if $\chi(S) > \chi(\Si)$, then $g(S) < g(\Si)$ and it would follow that $[S]=0$ in $H_2(\Si \times I; \ZZ/2)$. However, that would imply that $P$ and $N$ are $S^*$-equivalent. Since they are not, we must have that $\chi(S)=\chi(\Si),$ and this completes the proof of the claim.

\Cref{lemmaA} applies to show that $S$ is incompressible. As explained in \Cref{Cor-to-Waldhausen}, by \cite{Waldhausen-1968}, this implies that
$S$ is isotopic to $\Si\times \{t_0\}$ for some $0<t_0<1$. 
Under the isotopy, $L$ is isotopic to a link $L'$ that lies in $\Si\times I'$, where $I'$ is a closed interval containing $t_0$ in its interior. 

Projecting $L'$ along $\nu(S)\approx S\times I\to S$, gives a diagram $D$ for $L'$ which by the claim has genus $g(S)=g(\Si)$. Furthermore, the checkerboard surfaces of $D$ on $S$ are isotopic relative the boundary to $P$ and $N$. \Cref{prop-BW} now implies that $D$ is alternating. This finishes the proof.
\end{proof} 

\setcounter{section}{5} \noindent
\subsection{Epilogue}
We begin with a short proof of \Cref{thm:main} (from the introduction). We then state a corollary, and use it to deduce \Cref{cor:mainv} (also from the introduction). We end with a few closing remarks.
 
\smallskip\noindent
\textit{Proof of \Cref{thm:main}}.
Suppose $L \subset \Si \times I$ is a non-split link with positive and negative definite spanning surfaces $P$ and $N$, respectively. If $g(\Si)=0,$ then it is obvious that $L \subset \Si \times I$
has minimal genus. If $g(\Si)\geq 1,$ then arguing as above, we see that $P$ and $N$ are not $S^*$-equivalent.

Since $L$ admits spanning surfaces, it is checkerboard colorable. Choose the coloring so that $N$ is $S^*$-equivalent to $B$ and $P$ is $S^*$-equivalent to $W$, where $B$ and $W$ denote the black and white surfaces, respectively. The
Gordon-Litherland pairings $\cG_N$ and $\cG_P$ are evidently non-singular, and by \Cref{nonsingular}, so are $\cG_B$ and $\cG_W$. Therefore, \Cref{minimal-genus} implies that $L \subset \Si \times I$ must have minimal genus, and \Cref{thm:main2} implies that $L$ is alternating.
\qed \medskip \noindent

\begin{corollary}\label{thm:main2-intro}
A link $L\subset \Si\times I$ in a thickened surface is alternating and has minimal genus if and only if $L$ bounds  definite spanning
surfaces of opposite sign.
\end{corollary}

\Cref{cor:mainv} is an immediate consequence of \Cref{thm:main2-intro} and \Cref{cor-min-g}. 

In Theorem 1.2 of \cite{Greene}, Greene uses his characterization to deduce that any two connected, reduced, alternating diagrams of the same classical link have the same crossing number and writhe. A key result is Theorem 5.5 of \cite{Greene}, which shows that two connected bridgeless planar graphs with isometric flow lattices have the same number of edges. 
In this way, Greene gave a new geometric approach to establishing the first two Tait conjectures. 

Building on this approach, Kindred recently gave a geometric proof of the Tait flype conjecture \cite{Kindred-2020}. 
The first two Tait conjectures have been extended to alternating links in thickened surfaces and alternating virtual links in \cite{Boden-Karimi, Boden-Karimi-Sikora}. In \cite{Boden-Karimi}, the results are deduced using the homological Jones polynomial \cite{Krushkal-2011}. In \cite{Boden-Karimi-Sikora}, stronger statements are obtained using adequacy of the Kauffman skein bracket.

It is an open problem whether Greene and/or Kindred's methods can be extended to links in thickened surfaces. It would be interesting to use their approach to give alternative, geometric proofs of all three Tait conjectures in the generalized setting.

 \subsection*{Acknowledgements} This paper is based on several ideas in the Ph.D. thesis of the second author \cite{Karimi}, and it addresses a question raised by Liam Watson.
 We would like to thank him, as well as  Andrew Nicas and Will Rushworth, for their input and feedback. We would also like to thank the referee for their many insightful comments and suggestions for improvement.

\bibliographystyle{halpha}    
\newcommand{\etalchar}[1]{$^{#1}$}

\end{document}